\documentclass[12 pt,draft]{article}
\usepackage[cp1251]{inputenc}
\usepackage[russian]{babel}
\usepackage{amssymb,amsthm,amsmath}

\newtheorem{theorem}{Теорема}

\newtheorem{lemma}{Лемма}
\newtheorem{proposition}{Предложение}
\newtheorem{notation}{Замечание}

\begin{document}
\title{О рациональности многомерных возвратных степенных
рядов\footnote{Работа первого автора поддержана ..., работа
второго автора поддержана Российским фондом фундаментальных
исследований (РФФИ) и ...}}
\author{E.K. Лейнартас, А.П. Ляпин} \date{} \maketitle

А. Муавр рассмотрел под названием возвратных рядов степенные ряды
$F(z) = a_0 + a_1 z + ... + a_k z^k + ...$ с коэффициентами $a_1,
a_2, ..., a_k, ...$, образующими возвратные последовательности,
т.е. удовлетворяющими соотношению вида
$$
c_0 a_{m+p} + c_1 a_{m+p+1} + ... + c_m a_p = 0,\; p=0,1,2...,
$$
где $c_j$ -- некоторые постоянные. Оказалось, что такие ряды
всегда изображают  рациональные функции. Точнее, справедливо
следующее утверждение.\par\bigskip

\noindent \textbf{Теорема (Муавр, 1722).} {\itСтепенной ряд $F(z)$
является возвратным тогда и только тогда, когда он представляет
правильную рациональную функцию.}\par\bigskip

Доказательство этого факта можно найти, например, в
\cite{stanley1990}, в главе, посвященной простейшему общему классу
производящих функций -- рациональным функциям.

В многомерном случае ситуация значительно сложнее, например,
производящий ряд решения разностного уравнения с постоянными
коэффициентами в общем случае является расходящимся. Приведем
необходимые обозначения и определения и соответствующий пример.

Обозначим $x=(x_1, ..., x_n)$ точки $n$-мерной целочисленной
решетки $\mathbb Z^n = \mathbb Z \times ... \times \mathbb Z$, где
$\mathbb Z$ -- множество целых чисел, и $A=\{\alpha\}$ -- конечное
подмножество точек из $\mathbb Z^n$. Разностным уравнением
относительно неизвестной функции $f(x)$ целочисленных аргументов
$x=(x_1, ..., x_n)$ с постоянными коэффициентами $c_\alpha$
называется соотношение вида
\begin{align}\label{razn_ur}
\sum_{\alpha \in A} c_\alpha f(x+\alpha) =0.
\end{align}

В данной работе рассматривается случай, когда множество $A$ лежит
в положительном октанте $\mathbb Z_+^n$ целочисленной решетки и
удовлетворяет условию: существует точка $m= (m_1, ..., m_n)\in A$
такая, что для всех $\alpha \in A$ справедливы неравенства
\begin{align}\label{neravenstva}
\alpha_j \leqslant m_j, j= 1,2,...,n.
\end{align}

Обозначим через $P(z) = \sum_\alpha c_\alpha z^\alpha$
характеристический многочлен уравнения \eqref{razn_ur}. Пусть
$\delta_j$ -- оператор сдвига по переменной $x_j$: $\delta_j f(x)
= f(x_1, ..., x_{j-1}, x_j +1, x_{j+1}, ..., x_n)$, $\delta^\alpha
= \delta_1^{\alpha_1} \cdot \ldots \cdot \delta_n^{\alpha_n}$ и
под неравенством $\alpha\leqslant m$ будем понимать систему
неравенств $\alpha_j \leqslant m_j, j = 1, ..., n$. Тогда
разностное уравнение можно записать в виде:
\begin{align}\label{two}
P(\delta) f(x) \equiv \sum_{0\leqslant \alpha \leqslant m}
c_\alpha f(x+\alpha) = 0, x \in \mathbb Z_+^n.
\end{align}

Производящая функция ($z$-преобразование) функции $f(x)$
целочисленных аргументов $x\in \mathbb Z_+^n$ определяется
следующим образом:
\begin{align*}
F(z) = \sum_{x\geqslant 0} \frac {f(x)}{z^{x+I}}, \quad \text{ где
} I=(1,1,...,1).
\end{align*}

Если $f(x)$ -- решение разностного уравнения, то его производящая
функция $F(z)$ в случае $n>1$ представляет собой, вообще говоря,
расходящийся степенной ряд (см. \cite{lein2004}).\bigskip

\noindent \textbf{Пример}. Решением разностного уравнения
$$f(x_1+1, x_2+1) - f(x_1+1,x_2) - f(x_1, x_2+1) + f(x_1, x_2)=0$$
будет любая функция вида $f(x_1,x_2) = \varphi(x_1)+\psi(x_2)$,
где $\varphi$ и $\psi$ -- произвольные функции целочисленного
аргумента, а соответсвующий производящий ряд будет, вообще говоря,
расходящимся.\bigskip

Множество, на котором будем задавать <<начальные данные>>
разностного уравнения \eqref{two}, удовлетворяющего условию
\eqref{neravenstva}, определим следующим образом:
$$
X_0 = \{\tau \in \mathbb Z^n : \tau \geqslant 0, \tau \ngeqslant
m\},
$$
где символ $\ngeqslant$ означает, что точка $\tau$ лежит в
дополнении к множеству, определяемому системой неравенств $\tau_j
\geqslant m_j, j=1,..., n$.

Сформулируем задачу Коши: найти решение $f(x)$ уравнения
\eqref{two}, которое на множестве $X_0$ совпадает с заданной
функцией $\varphi(x):$
\begin{align}\label{nach_dan}
f(x) = \varphi(x), \; x\in X_0.
\end{align}

Нетрудно покзаать (см., например, \cite{lein2007}), что если
выполнено условие \eqref{neravenstva}, то задача
\eqref{two}-\eqref{nach_dan} имеет единственное решение. Вопрос о
разрешимости задачи \eqref{two}-\eqref{nach_dan} без ограничений
вида \eqref{neravenstva} рассмотрен в \cite{petrovsek}.

Приведем несколько формул, в которых производящая функция решения
задачи \eqref{two}-\eqref{nach_dan} выражается через начальные
данные.

Отметим, что с этого момента в данной работе будем рассматривать
такие задачи Коши для разностного уравнения, у которых
произвоядщий ряд \textbf{сходится} в некоторой окрестности
бесконечно удаленной точки.

\begin{theorem}\label{th_with_formulami}
Производящая функция $F(z)$ решения $f(x)$ задачи
\eqref{two}-\eqref{nach_dan} удовлетворяет следующим соотношениям:
\begin{gather}
P(z) F(z) = \sum_{0 \leqslant \alpha \leqslant m} c_\alpha \left(
\sum_{\substack{\tau \geqslant 0 \\ \tau \ngeqslant \alpha}} \frac
{\varphi(\tau)}{z^{\tau - \alpha + I}}\right);\label{formula1}\\
P(z) F(z) = \sum_{\substack{\tau \leqslant m \\ \tau \nleqslant
0}} \left( \sum_{\tau \leqslant \alpha \leqslant m} c_\alpha
\varphi(\alpha-\tau) \right) z^\tau;\label{formula2}\\ P(z)F(z) =
P(z) \sum_{\substack{\tau \geqslant 0 \\ \tau \ngeqslant m}} \frac
{\varphi(\tau)}{z^{\tau + I}} - \sum_{\substack{\tau \geqslant 0 \\
\tau \ngeqslant m}} \left( \sum_{0 \leqslant \alpha \leqslant
\tau} c_\alpha z^\alpha \right) \frac
{\varphi(\tau)}{z^{\tau+I}};\label{formula3}\\ P(z)F(z) =
\sum_{\substack{\tau \geqslant 0\\ \tau \ngeqslant m}} \left(
\sum_{\substack{\alpha \leqslant m \\ \alpha \nleqslant \tau}}
c_\alpha z^\alpha \right)
\frac{\varphi(\tau)}{z^{\tau+I}}.\label{formula4}
\end{gather}
\end{theorem}

\begin{notation}
В одномерном случае в правых частях всех формул
\eqref{formula1}-\eqref{formula4} стоят конечные суммы, в
частности, из формулы \eqref{formula2} следует, что произведение
$P(z) F(z)$ является многочленом. Таким образом, производящая
функция $F(z)$ является рациональной, что, собственно, и
доказывает теорему Муавра.
\end{notation}

Для формулировки многомерного варианта теоремы Муавра о
рациональности производящей функции решения разностного уравнения
нам потребуются следующие обозначения.

Пусть $J = (j_1, ..., j_n)$, где $j_k \in \{0,1\}, k=1,..., n,$ --
упорядоченный набор из нулей и единиц. Каждому такому набору
сопоставим <<грань целочисленного прямоугольника>> $\Pi_m =
\{x\in\mathbb Z^n: 0\leqslant x_k \leqslant m_k, k=1,...,n\}$
следующим образом:
$$\Gamma_J = \{x\in \Pi_m : x_k = m_k, \text{ если } j_k =1,
\text{ и } x_k < m_k, \text{ если } j_k =0\}.$$ Например,
$\Gamma_{(1,1, ..., 1)} = \{m\}$, a $\Gamma_{(0,0,..., 0)} =
\{x\in \mathbb Z^n: 0 \leqslant x_k < m_k, k=1, ..., n\}$.

Нетрудно проверить, что  $\Pi_m = \bigcup\limits_J \Gamma_J$ и для
различных $J, J'$ соответствующие грани не пересекаются: $\Gamma_J
\cap \Gamma_{J'} = \emptyset$.

Пусть $\Phi(z) = \sum\limits_{\substack{\tau \geqslant 0 \\
\tau \ngeqslant m}} \dfrac{\varphi(\tau)}{z^{\tau+I}}$ --
производящая функция начальных данных задачи
\eqref{razn_ur2}-\eqref{nach_dan} и каждой точке $\tau\in
\Gamma_J$ сопоставим ряд $\Phi_{\tau, J} (z) =
\sum\limits_{y\geqslant \; 0} \frac
{\varphi(\tau+Jy)}{z^{\tau+Jy+I}}$, а грани $\Gamma_J$ -- ряд
$\Phi_J (z) = \sum\limits_{\tau \in \; \Gamma_J} \Phi_{\tau,
J}(z)$. Если функцию $\varphi(x)$ начальных данных продолжить на
$\mathbb Z^n_+ \setminus X_0$ нулем, тогда производящую функцию
начальных данных можно записать в виде $\Phi(z) = \sum\limits_J
\Phi_J (z) = \sum\limits_J \sum\limits_{\tau\in \Gamma_J}
\Phi_{\tau,J}(z)$.

\begin{theorem}\label{th_o_proizv_f}
Производящая функция $F(z)$ решения задачи
\eqref{two}-\eqref{nach_dan} и производящая функция начальных
данных $\Phi(z)$ связаны соотношением
\begin{align}
P(z) F(z) = \sum_J \sum_{\tau\in \Gamma_J} \Phi_{\tau, J}(z)
P_\tau (z),
\end{align}
где многочлены $P_\tau (z)$ имеют вид $P_\tau(z) =
\sum\limits_{\substack{\alpha\leqslant m \\ \alpha \nleqslant
\tau}} c_\alpha z^\alpha$.
\end{theorem}

Отметим, что для случая $n=2$ теорема \ref{th_o_proizv_f} доказана
в работе \cite{Lyapin2009} в связи с изучением рациональных
последовательностей Риордана, в которой также приведен следующий
пример.\bigskip

\noindent \textbf{Пример}. Рассмотрим последовательности из $x$
элементов $a_1 a_2 ... a_x$, причем $a_1=0$ и $a_j\in\{0,1\}$ для
$2\leqslant j \leqslant x$. Элемент последовательности $a_j$
назовем изолированным, если он отличен от всех элементов, стоящих
на соседних местах. Обозначим $r(x,y)$ число таких
последовательностей, содержащих $y$ изолированных элементов.
Очевидно, что $r(x,y) =0$, если $x<y$.

Данная последовательность является решением задачи Коши для
разностного уравнения
$$
r(x+2, y+1) - r(x+1, y+1) - r(x+1, y) - r(x, y+1) + r(x,y) =0
$$
с начальными данными $\varphi(0,0) = 1, \varphi(1,0) = 0,
\varphi(x,0) = \varphi(x-1,0)+\varphi(x-2,0), x\geqslant 2$,
$\varphi(1,1)=1$, $\varphi(0,y) = 0, y\geqslant 1$ и $\varphi(1,y)
= 0, y\geqslant 2$.

Используя теорему 2, построим для каждого набора $J$ разбиение
прямоугольника $\Pi_{(2,1)}$ на множества: $\Gamma_{(0,0)} =
\{(0,0),(1,0)\}$, $\Gamma_{(0,1)} = \{(0,1),(1,1)\}$,
$\Gamma_{(1,0)} = \{(2,0)\}$, $\Gamma_{(1,1)} = \{(2,1)\}$.

Для каждого элемента множества $\Gamma_J$ построим ряды
$\Phi_{\tau, J}$ (ниже в обозначении $\Phi_{\tau, J}$ индекс $J$
опущен для краткости записи ) коэффициентами из соответствующих
начальных данных и многочлены $P_{\tau}$:
\begin{align*}
\Phi_{0,0}(z,w) = \frac 1{zw},\quad &P_{0,0}(z,w) = z^2w-zw-z-w, \\
\Phi_{1,0}(z,w) = 0,\quad &P_{1,0}(z,w) = z^2w-zw-w,\\
\Phi_{0,1}(z,w) = 0,\quad &P_{0,1}(z,w) = z^2w-zw-z,\\
\Phi_{1,1}(z,w) = \frac 1{z^2w^2},\quad &P_{1,1}(z,w) = z^2w,\\
\Phi_{2,0}(z,w) = \frac 1{zw(z^2-z-1)},\quad &P_{2,0}(z,w) =
z^2w-zw-w.
\end{align*}
Теперь, используя формулу из теоремы 2, легко записать
производящую функцию последовательности $r(n,k)$:
$$
F (z,w) = \frac {z - 1}{z^2w - zw - w - z + 1}.
$$
\bigskip

Из теоремы \ref{th_o_proizv_f} легко получается многомерный аналог
теоремы Муавра.

\begin{theorem}\label{th_moivre}
Производящая функция $F(z)$ решения задачи
\eqref{two}-\eqref{nach_dan} рациональна тогда и только тогда,
когда рациональна производящая функция $\Phi(z)$ начальных данных.
\end{theorem}

Приведем еще одно следствие из теоремы \ref{th_o_proizv_f}.
Дискретной функцией Грина $f_{\tau_0}(x), \tau_0 \in X_0,$
называется решение задачи \eqref{two}-\eqref{nach_dan} с
начальными данными
\begin{align}
\varphi_{\tau_0} (x) =
\begin{cases}
1, \text{ если } x = \tau_0,\\
0, \text{ если } x \neq \tau_0.
\end{cases}
\end{align}

Отметим, что асимптотические свойства дискретной функции Грина
играют важную роль при исследовании устойчивости линейной
однородной двуслойной разностной схемы (см., например,
\cite{fedo}).

\begin{proposition}
Производящая функция $F_{\tau_0}(z)$ дискретной функции Грина
$f_{\tau_0}(x)$ задачи \eqref{two}-\eqref{nach_dan} рациональна и
имеет вид
$$
F_{\tau_0} (z) = \left( \sum_{\substack{\alpha\leqslant m \\
\alpha \nleqslant \tau_0}} c_\alpha z^{\alpha - \tau_0 - I}\right)
\slash P(z).
$$
\end{proposition}

Приведем доказательства сформулированных утверждений.

\begin{proof}[Доказательство теоремы \ref{th_with_formulami}]
Умножим ряд $F(z) = \sum\limits _{x\geqslant 0} \frac
{f(x)}{x^{x+I}}$ на характеристический многочлен $P(z) =
\sum\limits_\alpha c_\alpha z^\alpha$ и после перегруппировки с
учетом уравнения \eqref{two} получим
\begin{align*}
P(z) F(z) = \left( \sum_\alpha c_\alpha z^\alpha \right) \left(
\sum_{x\geqslant 0} \frac{f(x)}{z^{x+I}}\right)  =
\left(\sum_\alpha c_\alpha z^\alpha\right) \left( \sum_{x\geqslant
\alpha} \frac{f(x)}{z^{x+I}} + \sum_{x\ngeqslant \alpha}
\frac{f(x)}{z^{x+I}} \right) =\\= \sum_\alpha \left(c_\alpha
\sum_{x\geqslant \alpha} \frac{f(x)}{z^{x-\alpha+I}}\right) +
\sum_\alpha \left(c_\alpha z^\alpha \sum_{x \ngeqslant\: \alpha}
\frac{f(x)}{z^{x+I}}\right) =\\= \sum_\alpha \left(c_\alpha
\sum_{x \geqslant 0} \frac {f(x+\alpha)}{z^{x+I}}\right) +
\sum_\alpha
c_\alpha z^\alpha \sum_{x \ngeqslant \:\alpha} \frac {f(x)}{z^{x+I}} = \\
= \sum_{x\geqslant 0} \frac {P(\delta)f(x)}{z^{x+I}} + \sum_\alpha
c_\alpha z^\alpha \sum_{x \ngeqslant \alpha} \frac {f(x)}{z^{x+I}}
= \sum_\alpha c_\alpha z^\alpha \sum_{x \ngeqslant \alpha} \frac
{f(x)}{x^{x+I}}.
\end{align*}
Формула \eqref{formula1} доказана. Формулы \eqref{formula2} и
\eqref{formula3} получаются из формулы \eqref{formula1}
группировкой слагаемых в правой части соответственно относительно
$z^\tau$ и $\varphi(\tau)$. Формула \eqref{formula4} есть простое
следствие формулы \eqref{formula3}.
\end{proof}

\begin{proof}[Доказательство теоремы \ref{th_o_proizv_f}]
Воспользуемся формулой \eqref{formula4}:
\begin{align*}
P(z) F(z) = \sum_{\substack{\tau \geqslant\; 0 \\ \tau
\ngeqslant\; m}} P_\tau (z) \frac {\varphi(\tau)}{z^{\tau+I}} =
\sum_J \sum_{\tau \in\; \Gamma_J} \sum_{y\geqslant\; 0} P_{\tau +
J y} (z) \cdot \frac {\varphi(\tau+Jy)}{z^{\tau+Jy+I}}.
\end{align*}
Из определения многочлена $P_\tau(z)$ следует, что для любого
$\tau \in \Gamma_J$ и любого $y\geqslant 0$ справедливо равенство
$P_{\tau+Jy}(z) = P_\tau (z)$, а используя определение
$\Phi_{\tau, J}(z)$, получим
\begin{equation*}
P(z) F(z) = \sum_J \sum_{\tau \in \;\Gamma_J} P_\tau(z)
\Phi_{\tau,\; J} (z).
\end{equation*}
\end{proof}

Для доказательства теоремы \ref{th_moivre} нам потребуются
следующие вспомогательные утверждения.

\begin{lemma}\label{lemma1}
$1^0$. Пусть степенной ряд $\Phi(\xi) = \sum\limits_{x\geqslant 0}
\varphi(x) z^x$ сходится в некоторой окрестности $U =\{\xi:
|\xi_j| < \rho_j, j=1,...,n\}$ начала координат. Тогда для любого
$\tau\in \mathbb Z^n_+$ и любого набора $J = \{j_1,..., j_n\}, j_k
\in \{0,1\}$ для функции $\Phi_{\tau, J} (z) =
\sum\limits_{y\geqslant 0 }\varphi(\tau + J y) z^{\tau + Jy}$
справедлива интегральная формула
\begin{align*}
\Phi_{\tau, J} (z) = \frac{z^\tau}{(2\pi i)^n} \int\limits_\Gamma
\frac{\Phi(\xi)d\xi}{(\xi-z)^J \xi^{\tau - J +I}},
\end{align*}
где $\Gamma = \{\xi\in \mathbb C^n: |\xi_j| = R_j < \rho_j,
j=1,...,n\}, (\xi - z)^J = (\xi_1 - z_1)^{j_1} \cdot ... \cdot
(\xi_n - z_n)^{j_n}$. Формула справедлива для таких $z$, что
$|z_j| < R_j, j=1,...,n$.\par

$2^0$. Если функция $\Phi(z)$ рациональна, то функция $\Phi_{\tau,
J}(z)$ также рациональна.
\end{lemma}

\begin{proof}[Доказательство леммы]
Для доказательства интегрального представления функции разложим
$\Phi_{\tau, J}(z)$ подынтегральную функцию в степенной ряд
$$\frac{\Phi(\xi)}{(\xi-z)^J} \cdot \frac 1{\xi^{\tau-J+I}} =
\sum_{x\geqslant \,0} \varphi(x)\xi^x \cdot \sum_{y\geqslant \, 0}
\frac{z^{Jy}}{\xi^{Jy}}\cdot \frac 1{\xi^{\tau+I}},$$ сходящийся
абсолютно и равномерно на множестве $\Gamma$, и почленно
проинтегрируем его.\par

Для доказательства рациональности функции $\Phi_{\tau, J} (z)$
достаточно применить к полученному выражению для функции через
интеграл повторное интегрирование и воспользоваться теоремой о
вычетах. Так, при интегрировании по первой переменной нужно будет
найти вычеты в полюсах $\xi=0$ и $\xi=z_1$ (если $j_1=1$). При
этом важно, что рациональность подынтегральной функции сохраняется
при каждом интегрировании.
\end{proof}

\begin{notation}
Утверждение $2^0$ леммы также можно доказать, используя понятие
сечения кратного степенного ряда (см. \cite{Duffin}) и теорему о
рациональности сечения ряда, представляющего собой рациональную
функцию.
\end{notation}

\begin{notation}
Если вместо степенного ряда $\Phi(z) = \sum\limits_{x\geqslant\:
0} \varphi(x) z^x$ рассматривать ряд Лорана $\Phi(\xi) =
\sum\limits_{x\geqslant\:0} \frac {\varphi(x)}{z^{x+I}}$ и,
соответственно, функцию $\Phi_{\tau, J}(z) = \sum_{y\geqslant \:0}
\frac{\varphi(\tau+Jy)}{z^{\tau+Jy+I}}$, то утверждение $2^0$
леммы \ref{lemma1} о ее рациональности остается справедливым.
\end{notation}

\begin{proof}[Доказательство теоремы \ref{th_moivre}] Пусть
производящая функция $F(z) = \sum\limits_{x\geqslant 0}
\frac{f(x)}{z^{x+I}}$ решения задачи (3)-(4) для разностного
уравнения рациональна и $\Phi(z) =
\sum\limits_{\substack{x\geqslant\:0\\x \ngeqslant\: m}}
\frac{\varphi(x)}{z^{x+I}}$ -- производящая функция начальных
данных. В силу леммы (а точнее, замечания к ней) функция
$\Phi_{m,J}(z) = \sum\limits_{x\geqslant\: 0}
\frac{\varphi(x)}{z^{x+I}}$ рациональна (для $J=(1, ..., 1)$), но
тогда $\Phi(z) = F(z) - \Phi_{m,J}(z)$, т.е. $\Phi(z)$ --
рациональна.

Для доказательства достаточности воспользуемся теоремой
\ref{th_o_proizv_f}, согласно которой для производящей функции
$F(z)$ решения задачи (3)-(4) справедливо соотношение $$P(z) F(z)
= \sum\limits_J \sum\limits_{\tau\in \Gamma_J} \Phi_{\tau, J}
P_\tau (z),$$ где $P_\tau(z)$ -- многочлены. Если производящая
функция $\Phi(z)$ начальных данных рациональна, то в силу леммы
будут рациональными и функции $\Phi_{\tau, J}(z)$. Таким образом
производящая функция $F(z)$ решения задачи (3)-(4) рациональна.

\end{proof}

Доказательство предложения 1 сразу следует из формулы
\eqref{formula4}.

\end{document}